\documentclass[letter, reqno,14pt]{amsart}

\usepackage[usenames,dvipsnames]{color}
\usepackage[utf8]{inputenc}
\usepackage{amsthm,amsfonts,amssymb,amsmath,amsxtra}
\usepackage[all]{xy}
\SelectTips{cm}{}

\usepackage{nccrules}

\usepackage{verbatim}
\usepackage{xr-hyper}
\usepackage{varioref}
\usepackage{hyperref}
\usepackage[nameinlink]{cleveref}
\usepackage{tikz-cd}
\usepackage[margin=1.25in]{geometry}
\usepackage{mathrsfs}
\usepackage{xcolor}

\usepackage[shortlabels]{enumitem}
\usepackage{lipsum}
\usetikzlibrary{positioning,chains,fit,shapes,calc,arrows,patterns,external,shapes.callouts,graphs}

\usepackage[style=alphabetic]{biblatex}
\RequirePackage{xspace}
\RequirePackage{etoolbox}
\RequirePackage{varwidth}
\RequirePackage{enumitem}
\RequirePackage{tensor}
\RequirePackage{mathtools}
\RequirePackage{longtable}
\RequirePackage{multirow}

\def\ge{\geqslant}
\def\le{\leqslant}
\def\a{\alpha}

\def\g{\gamma}
\def\G{\Gamma}

\def\D{\Delta}

\def\s{\sigma}
\def\t{\tau}

\def\k{\kappa}
\def\l{\lambda}
\def\z{\zeta}

\def\i{^{-1}}

\def\<{\langle}
\def\>{\rangle}

\newcommand{\bG}{\mathbf G}

\newcommand{{\BG}}{\ensuremath{\mathbb {G}}\xspace}

\newcommand{{\BK}}{\ensuremath{\mathbb {K}}\xspace}

\newcommand{\BQ}{\ensuremath{\mathbb {Q}}\xspace}

\newcommand{\BS}{\ensuremath{\mathbb {S}}\xspace}

\newcommand{\BZ}{\ensuremath{\mathbb {Z}}\xspace}

\newcommand{\CI}{\ensuremath{\mathcal {I}}\xspace}

\newcommand{\CO}{\ensuremath{\mathcal {O}}\xspace}

\DeclareMathOperator{\Aut}{Aut}
\newcommand{\wt}{\text{wt}}
\newcommand{\ra}{\rightarrow}

\def\tW{\tilde W}

\DeclareMathOperator{\Adm}{Adm}

\DeclareMathOperator{\Gal}{Gal}

\newcommand{\wtd}{\widetilde}

\newcommand{\ad}{\text{ad}}

\newtheorem{theorem}{Theorem}
\newtheorem{proposition}[theorem]{Proposition}
\newtheorem{lemma}[theorem]{Lemma}

\newtheorem{corollary}[theorem]{Corollary}

\theoremstyle{definition}
\newtheorem*{acknowledgement}{Acknowledgement}

\newtheorem{remark}[theorem]{Remark}

\numberwithin{equation}{section}
\numberwithin{theorem}{section}

\setitemize[0]{leftmargin=*,itemsep=\the\smallskipamount}
\setenumerate[0]{leftmargin=*,itemsep=\the\smallskipamount}

\renewcommand{\to}{%
   \ifbool{@display}{\longrightarrow}{\rightarrow}%
   }
\let\shortmapsto\mapsto
\renewcommand{\mapsto}{%
   \ifbool{@display}{\longmapsto}{\shortmapsto}%
   }
\newlength{\olen}
\newlength{\ulen}
\newlength{\xlen}
\newcommand{\xra}[2][]{%
   \ifbool{@display}%
      {\settowidth{\olen}{$\overset{#2}{\longrightarrow}$}%
       \settowidth{\ulen}{$\underset{#1}{\longrightarrow}$}%
       \settowidth{\xlen}{$\xrightarrow[#1]{#2}$}%
       \ifdimgreater{\olen}{\xlen}%
          {\underset{#1}{\overset{#2}{\longrightarrow}}}%
          {\ifdimgreater{\ulen}{\xlen}%
             {\underset{#1}{\overset{#2}{\longrightarrow}}}
             {\xrightarrow[#1]{#2}}}}%
      {\xrightarrow[#1]{#2}}
   }
\makeatother
\newcommand{\xyra}[2][]{%
   \settowidth{\xlen}{$\xrightarrow[#1]{#2}$}%
   \ifbool{@display}%
      {\settowidth{\olen}{$\overset{#2}{\longrightarrow}$}%
       \settowidth{\ulen}{$\underset{#1}{\longrightarrow}$}%
       \ifdimgreater{\olen}{\xlen}%
          {\mathrel{\xymatrix@M=.12ex@C=3.2ex{\ar[r]^-{#2}_-{#1} &}}}%
          {\ifdimgreater{\ulen}{\xlen}%
             {\mathrel{\xymatrix@M=.12ex@C=3.2ex{\ar[r]^-{#2}_-{#1} &}}}
             {\mathrel{\xymatrix@M=.12ex@C=\the\xlen{\ar[r]^-{#2}_-{#1} &}}}}}%
      {\mathrel{\xymatrix@M=.12ex@C=\the\xlen{\ar[r]^-{#2}_-{#1} &}}}%
   }
\makeatletter
\newcommand{\xla}[2][]{%
   \ifbool{@display}%
      {\settowidth{\olen}{$\overset{#2}{\longleftarrow}$}%
       \settowidth{\ulen}{$\underset{#1}{\longleftarrow}$}%
       \settowidth{\xlen}{$\xleftarrow[#1]{#2}$}%
       \ifdimgreater{\olen}{\xlen}%
          {\underset{#1}{\overset{#2}{\longleftarrow}}}%
          {\ifdimgreater{\ulen}{\xlen}%
             {\underset{#1}{\overset{#2}{\longleftarrow}}}
             {\xleftarrow[#1]{#2}}}}%
      {\xleftarrow[#1]{#2}}
   }
\newcommand{\isoarrow}{%
   \ifbool{@display}{\overset{\sim}{\longrightarrow}}{\xrightarrow\sim}%
   }
 
\begin{document}

\title[On the dimension of some union of affine Deligne-Lusztig varieties]{On the dimension of some union of affine Deligne-Lusztig varieties}

\author{Arghya Sadhukhan}
\address[A. S.]{Department of Mathematics, University of Maryland, College Park, MD 20742}

\email{arghyas0@math.umd.edu}

\begin{abstract}
In this paper, we consider certain union $X(\mu,b)$ of affine Deligne-Lusztig varieties in the affine flag variety that arises in the study of mod-$p$ reduction of Rapoport-Zink spaces and moduli spaces of shtukas associated to a connected reductive group. Under a mild hypothesis on $\mu$, but no further restrictions on the group, we compute its dimension in the case where $b$ is the maximal neutrally acceptable element. 
\end{abstract}

\maketitle

\vspace{-0.8cm}
\section{Introduction}
\subsection{Motivation}
A central tool in understanding the reduction modulo $p$ of Shimura varieties, and therefore in the study of arithmetic properties encoded in them, is the Newton stratification. Roughly speaking, at least for the Shimura varieties of Hodge type that admits a moduli interpretation, this stratification on the special fiber (of a suitable integral model) is defined by the loci where the $p$-divisible groups attached to the corresponding abelian varieties have fixed isogeny type. Investigation into the unique closed Newton stratum, the so-called basic locus, in the last decades has found fruitful applications in the realm of the Langlands program. We mention the work toward Zhang's arithmetic fundamental lemma \cite{RTZ13}, as well as the theory of non-archimedean uniformization of Shimura varieties, first pioneered by Cerednik and Drinfeld, and further developed by Rapoport–Zink \cite{RZ}, Fargues \cite{Far}, and Howard–Pappas \cite{HP17}.

%to compute intersection numbers of special cycles as in the Kudla-Rapoport program \cite{KL11}, in the work towards Zhang’s arithmetic fundamental lemma \cite{RTZ13}, or in the work on Tate conjecture for certain Shimura varieties \cite{HTX}.

%for instance, in the theory of non-archimedean uniformization of Shimura varieties, as pioneered by Čerednik and Drinfeld, and further developed by Rapoport–Zink, Fargues, Kim, and Howard–Pappas 

%The non-basic strata, on the other hand, can sometimes be related to Levi subgroups of the underlying algebraic group. One
%can then study those so-called Hodge-Newton decomposable strata (and their cohomology, etc.) by a kind of induction process, cf. \cite{GHN16}.

The natural range of the index set for the Newton stratification corresponding to a Shimura variety associated to the group-theoretic data $(\bG,\{\mu\})$ is the set $B(\bG, \mu)$ of neutral acceptable elements. Using the classification of the Kottwitz set obtained in \cite{Ko85} and \cite{Ko97}, $B(\bG, \mu)$ can be viewed as a subset of the dominant rational (relative) Weyl chamber.
Thus it becomes a partially ordered set with respect to the usual dominance relation. It is easy to see from the definition that $B(\bG,\mu)$ has a unique minimal element; the basic locus corresponds to this minimal element. A nontrivial result due to He and Nie \cite{HN18} asserts that there is a unique maximal element in $B(\bG,\mu)$; to give an example, when the $\mu$-ordinary locus is nonempty, it is the unique maximal Newton stratum, and then the maximal element of $B(\bG,\mu)$ is $\mu^\diamond$, the twisted Galois-average (see \cref{sec:bgmu}) of a dominant representative in $\{\mu\}$. However, the description of this maximal element is rather mysterious in general.

The primary goal of this article is to investigate the dimension of a certain union of affine Deligne-Lusztig varieties associated to this maximal element. Such a geometric object is the group-theoretic model for the maximal Newton stratum. We remark that in the situations where the He-Rapoport axioms in \cite{HR17} hold, this sought-after dimension is equal to the dimension of the maximal Newton stratum minus the dimension of the corresponding central leaf, see \cite[\S 2.12]{He16}. The first quantity is $\<2\rho,\mu\>$, where $2\rho$ is the sum of the (relative) positive roots, so in that case, we essentially want to compute the latter dimension. The said (union of) affine Deligne-Lusztig varieties can be defined purely in terms of group-theoretic data without any reference to Shimura varieties, which we briefly do in the following subsection. In the remainder of the article, we resort to this latter setup.

\subsection{Main theorem} Let $\bG$ be a connected reductive group over a non-archimedean local field $F$. Denote by $\s$ the Frobenius of $\breve{F}$ over $F$, where $\breve{F}$ is the completion of the maximal unramified extension of $F$. We choose a $\s$-invariant Iwahori subgroup $\breve{I}$ of $\bG(\breve{F})$ . Fix a geometric
conjugacy class $\{\mu\}$ of cocharacters of $\bG$ over the separable closure $\bar{F}$ and choose a dominant representative $\mu$ in it. We denote
by $\Adm(\mu)$ the $\mu$-admissible set, a finite subset of the Iwahori-Weyl group $\wtd{W}$, see \cref{sec:pre} for details. Then for an element $b\in \bG(\breve{F})$, we define following Rapoport \cite{Rap05} the union of affine Deligne-Lusztig varieties 
$$X(\mu,b)=\{g \breve \CI \in \bG(\breve F)/\breve \CI; g \i b \s(g) \in \breve \CI \Adm(\mu) \breve \CI\}.$$
This is the set of geometric points of a subscheme, locally of finite type, of the affine flag variety attached to $\breve{F}$ (in the usual
sense in equal characteristic; in the sense of Zhu \cite{Zhu17} and Bhatt-Scholze \cite{BS17} in mixed characteristics). The interest in such a set comes from the fact that when $F$ is $p$-adic and $\mu$ is minuscule, sets of this form arise as the set of geometric points of the underlying reduced set of a
Rapoport-Zink formal moduli space of $p$-divisible groups, see \cite{RV14}. Something analogous holds
in the function field case for formal moduli spaces of shtukas, where the
minuscule hypothesis can be dropped, cf. \cite{Vi18}. In either case, by He's work on the Kottwitz-Rapoport conjecture in \cite{He16a}, the set $B(\bG,\mu)$ can be characterized as the set of those $\s$-conjugacy classes $[b]$ such that $X(\mu,b)\neq \emptyset$. Our main result is then the following.
\begin{theorem}\label{thm}
Assume that the image $\underline \mu_\ad$ of $\mu$ in $X_*(T_\ad)_{\G_0}$ has depth at least $2$ in every $\bar{F}$-simple factor of $\bG_\ad$. Let $b=b_{\max}$ be the maximal element of $B(\bG,\mu)$. Then
\begin{equation}\label{main}
    \dim X(\mu,b)=\text{rk}_F^{\text{ss}} \bG^* - \text{rk}_F^{\text{ss}} \bG,
\end{equation}
where $\bG^*$ is the unique (up to isomorphism) quasi-split inner form of $\bG$, and $\text{rk}_F^{\text{ss}}$ stands for the semisimple $F$-rank.
\end{theorem}
Here \textit{depth} of a dominant cocharacter is an estimate of its distance from the walls of the dominant Weyl chamber, see \cref{product} for a precise definition. We remark that, in particular, $\mu$ cannot be minuscule because of this restriction on its depth. Hence, by way of application in arithmetic geometry, this result is only of interest in the function field setup.  

Let us point out that the quantity on the RHS of \cref{main} is a non-negative integer since every torus in $\bG$ transfers to $\bG^*$, cf. \cite[lemma 2.1]{Lan89}, and it is zero precisely when $\bG$ is quasi-split. On the other hand, by applying Chai's description of $B(\bG,\mu)$ in \cite[\S 7.1]{Ch00}, \cite[lemma 2.5]{HN18}, we can see that the maximal element $b_{\max}$ has Newton point $\mu^\diamond$ exactly when $\bG$ is quasi-split; also cf. \cite[proposition 5.5]{GHN20} for a different argument. As a result, the dimension on the LHS above can be directly checked to be zero in the quasi-split case. Hence, \cref{main} indeed recovers this trivial case (which is valid even in the absence of the depth hypothesis), and as such this theorem also serves as a geometric way to measure how far $\bG$ is from being quasi-split. 

\subsection{Context of the result and outline of the proof}
Let us put \cref{thm} into context. For a quasi-split group $\bG$, a dimension formula on $X(\mu,b)$ for general $b$ was first obtained by He and Yu \cite{HY21} under a 
superregularity hypothesis on the depth of $\mu$. In \cite[theorem C]{Sad21}, the same formula was shown to be valid under the weaker assumption that $\mu$ is regular, i.e., $\text{depth}(\mu)\geq 1$. Let us now place ourselves in the context of a general group $\bG$, and let $b_{\min}$ be the unique minimal (equivalently, basic) element of $B(\bG,\mu)$. Then the work of G{\"o}rtz, He and Rapoport in \cite{GHR22} gives characterizations of when $X(\mu,b_{\min})$ can be of minimal dimension zero, or of maximal dimension $\<2\rho,\mu\>$; the model cases of such phenomena are the Lubin-Tate case and the Drinfeld case, respectively. Let us also remark that in \cite{GHN20}, G{\"o}rtz, He and Nie provide a lower bound for the dimension of $X(\mu,b_{\min})$, and they are also able to classify completely when this becomes an equality. Besides these results, as far as the author is aware, there is no dimension formula available for $X(\mu,b)$, nor is there any conjectural description. Standing at this juncture, it is therefore natural to investigate the dimension problem for the other extremal element of $B(\bG,\mu)$, namely $b_{\max}$.

It is worth pointing out that our dimension formula can also be obtained more directly, using explicit knowledge of the element $b_{\max}$ as discussed in \cref{b-max-compute}. Our method, on the other hand, has the merit of producing all the elements, at least conjecturally, in the Iwahori-Weyl group whose associated Iwahori cells contribute to the maximal Newton stratum; such description would be useful to analyze the (top dimensional) irreducible components of $X(\mu,b_{\max})$ in the non quasi-split setup. 

We now briefly sketch the strategy of the proof. We first identify $b_{\max}$ as the maximum of generic $\s$-conjugacy classes associated with the maximal translation elements in the $\mu$-admissible set. Then we proceed to express such generic $\s$-conjugacy classes in two distinct ways, first in terms of $\s$-twisted Demazure power and then via a reduction to the quasi-split case. Such considerations allow us to express the dimension of individual affine Deligne-Lusztig varieties associated to such generic $\s$-conjugacy classes in terms of certain statistics on the quantum Bruhat graph introduced by Brenti, Fomin and Postnikov \cite{BFP99}. Further analysis in \cref{sec:finish} shows that $\dim X(\mu,b_{\max})$ lies between the length and the reflection length of certain minimal length elements of some Frobenius-twisted conjugacy classes in the finite Weyl group. We then show that such minimal length elements are, in fact, partial Coxeter elements via a case-by-case calculation in \cref{sec:minlength}, and hence their length (which is equal to their reflection length) matches the asserted dimension in \cref{thm}, thereby finishing the proof.

\begin{acknowledgement}
The author thanks Xuhua He for his suggestion to work on this problem and for many useful remarks throughout the course of this work. He would also like to thank Michael Rapoport for his interest in this work and his suggestions on a preliminary draft, and Mita Banik for her help with SageMath computations.
\end{acknowledgement}

\section{Preliminaries}\label{sec:pre}
\subsection{Notations}\label{sec:notation} 

Recall that $\bG$ is a connected reductive group over a non-archimedean local field $F$. Let $\breve F$ be the completion of the maximal unramified extension of $F$ and $\s$ be the Frobenius morphism of $\breve F/F$. The residue field of $F$ is a finite field $\mathbb F_q$ and the residue field of $\breve F$ is the algebraically closed field $\bar{\mathbb F}_q$. We write $\breve G$ for $\bG(\breve F)$. We use the same symbol $\s$ for the induced Frobenius morphism on $\breve G$. Let $S$ be a maximal $\breve F$-split torus of $\bG$ defined over $F$, which contains a maximal $F$-split torus. Let $\mathcal A$ be the apartment of $\bG_{\breve F}$ corresponding to $S_{\breve F}$. We fix a $\s$-stable alcove $\mathfrak a$ in $\mathcal A$, and let $\breve \CI \subset \breve G$ be the Iwahori subgroup corresponding to $\mathfrak a$. Then $\breve \CI$ is $\s$-stable.

Let $T$ be the centralizer of $S$ in $\bG$. Then $T$ is a maximal torus. We denote by $N$ the normalizer of $T$ in $\bG$. The \emph{Iwahori--Weyl group} (associated to $S$) is defined as $$\wtd{W}= N(\breve F)/T(\breve F) \cap \breve \CI.$$ For any $w \in \wtd{W}$, we choose a representative $ \dot{w}$ in $N(\breve F)$; however if there is no possibility of confusion we will call the lift $w$ too. The action $\s$ on $\breve G$ induces a natural action of $\s$ on $\wtd{W}$, which we still denote by $\s$. 

We denote by $\ell$ the length function on $\tW$ determined by the base alcove $\mathfrak a$ and denote by $\tilde \BS$ the set of simple reflections in $\tW$. Let $W_{\text{aff}}$ be the subgroup of $\wtd{W}$ generated by $\tilde \BS$. Then $W_{\text{aff}}$ is an affine Weyl group. Let $\Omega \subset \wtd{W}$ be the subgroup of length-zero elements in $\wtd{W}$. Then $$\wtd{W}=W_{\text{aff}} \rtimes \Omega.$$ Since the length function is compatible with the $\s$-action, the semi-direct product decomposition $\wtd{W}=W_{\text{aff}} \rtimes \Omega$ is also stable under the action of $\s$.  

Let $W=N(\breve F)/T(\breve F)$ be the relative Weyl group. We denote by $\BS$ the subset of $\wtd{\BS}$ consisting of simple reflections generating $W$. We let $\Phi$ (resp. $\D$) denote the set of roots (resp. simple roots) for $W$. We write $\Gamma$ for $\Gal(\bar F/F)$, and write $\Gamma_0$ for the inertia subgroup of $\Gamma$. Then fixing a special vertex of the base alcove $\mathfrak a$, we have the splitting $$\wtd{W}=X_*(T)_{\G_0} \rtimes W=\{t^{\underline \l} w; \underline \l \in X_*(T)_{\G_0}, w \in W\}.$$ Note that if $\bG$ is not quasi-split over $F$, then there does not exist a $\s$-stable special vertex in $\mathfrak a$ and thus the splitting $\wtd{W}=X_*(T)_{\G_0} \rtimes W$ is not $\s$-stable. 

For an irreducible Weyl group $W$ of rank $n$, we follow the labeling of roots as in \cite{Bou} and we usually write $s_i$ instead of $s_{\a_i}$, where $\D=\{\a_i: 1 \leq i \leq n\}$. Let $w_0$ be the longest element in $W$, and for $i \in [1,n]$ we let $w_{i,0}$ be the longest element of the parabolic subgroup of $W$ corresponding to $\D\setminus \{\a_i\}$. Let $\rho$ be the dominant weight with $\<\a^\vee, \rho\>=1$ for any $\a \in \D$. Let $\{\varpi_i^\vee: 1\leq i\leq n\}$ be the set of fundamental coweights. If $\varpi^\vee_i$ is minuscule,
we denote the image of $t^{\underline \varpi_i^\vee}$ under the projection $\wtd{W} \ra \Omega$ by $\t_i$; then conjugation by $\t_i$ is a length preserving automorphism of $\wtd{W}$ which we denote by $\text{Ad}(\t_i)$. 

%Let $\tilde \BS$ be the set of simple reflections in $\tW$ and $\BS \subset \tilde \BS$ be the set of simple reflections in $W_0$. The action of $\s$ on $\tW$ induces an action on the set $\tilde \BS$, which we still denote by $\s$. However, in general, $\BS$ is not stable under the action of $\s$. 
\subsection{Quantum Bruhat graphs}
We recall the quantum Bruhat graph introduced by Brenti, Fomin, and Postnikov in the context of quantum cohomology ring of complex flag variety, see \cite{BFP99}, also \cite{FGP97}. By definition, a \emph{quantum Bruhat graph} $\G_{\Phi}$ is a directed graph with 
\begin{itemize}
\item vertices given by the elements of $W$; 

\item upward edges $w\rightharpoonup w s_\a$ for some $\a \in \Phi^+$ with $\ell(w s_\a)=\ell(w)+1$;

\item downward edges $w \rightharpoondown w s_\a$ for some $\a \in \Phi^+$ with $\ell(w s_\a)=\ell(w)-\< 2 \rho, \a^\vee\>+1$. 
\end{itemize}

\medskip
%We now recall some graph theoretic notions related to $\G_\Phi$; we refer to \cite[\S 3.1]{Mil21} for more details. For two elements $w, w'$, a path between them is defined to be concatenation of edges joining a sequence of vertices in $\G_\Phi$. 
The \emph{weight of a path} in $\G_{\Phi}$ is defined to be the sum of weights of the edges in the path. For any $x,y\in W$, we denote by $d(x, y)$ the minimal length among all paths in $\G_{\Phi}$ from $x$ to $y$. Any path between $x$ and $y$ affording $d(x,y)$ as its length is called a \textit{shortest path} between them. Then the following lemma defines a function $\text{wt}: W \times W \ra \BZ_{\geq 0} \Phi^\vee$.
%The length of path is defined to be the number of edges present in it. For any $x,y\in W$, we denote by $d_\G(x, y)$ the minimal length among all paths in $\G_{\Phi}$ from $x$ to $y$. Any path between $x$ and $y$ affording $d_\G(x,y)$ as its length is called a \textit{shortest path} between them. Then the following lemma defines a function $\text{wt}: W \times W \ra \BZ_{\geq 0} \Phi^\vee$.

\begin{lemma}\label{wt-x-y} \cite[Lemma 1]{Pos05}, \cite[Lemma 6.7]{BFP99}
Let $x,y\in W$. 
\begin{enumerate}
\item There exists a directed path (consisting of possibly both upward and downward edges) in $\G_{\Phi}$ from $x$ to $y$.

\item Any two shortest paths in $\G_{\Phi}$ from $x$ to $y$ have the same weight, which we denote by $\text{wt}(x, y)$. 

\item Any path in $\G_{\Phi}$ from $x$ to $y$ has weight $\ge \wt(x,y)$.
\end{enumerate}
\end{lemma}
We need the following results, taken from \cite[\S 4.2]{HY21} and \cite[\S 2.5]{HN21}.
\begin{lemma}\label{wt-d}
Let $x,y \in W$. Then
\begin{enumerate}
    \item We have $\<\wt(a,b),\a\> \leq 2$, for any simple root $\a$.
    \item $\ell(y)-\ell(x)=d(x,y)-\<2\rho,\wt(x,y)\>$.
\end{enumerate}

\end{lemma}
\subsection{Demazure product} \label{product}
We follow \cite[\S 1]{He09}. For any $x, y \in \wtd{W}$, the subset $\{x y'; y' \le y\}$ (resp. $\{x' y; x' \le x\}$, $\{x' y'; x' \le x, y' \le y\}$) of $\wtd{W}$ contains a unique maximal element with respect to the Bruhat order. Moreover, we have $$\max\{x y'; y' \le y\}=\max\{x' y; x' \le x\}=\max\{x' y'; x' \le x, y' \le y\}.$$ We denote this element by $x \ast y$ and we call it the {\it Demazure product} of $x$ and $y$. Then $(\wtd{W}, \ast)$ is a monoid and the Demazure product is determined by the following two rules
\begin{itemize}
    \item $x \ast y = x y$ if $x, y \in \wtd{W}$ such that $\ell(x y) = \ell(x) + \ell(y)$;
    \item $s \ast w = w$ if $s \in \tilde \BS$, $w \in \wtd{W}$ such that $s w < w$.
\end{itemize}

The Demazure product occurs naturally in the convolution product of Iwahori cells in the sense that for any two elements $x,y \in \wtd{W}$, we have $\breve{I}\dot{x}\breve{I} \ast \breve{I}\dot{y}\breve{I}=\breve{I}(\dot{x} \ast \dot{y})\breve{I}$. By such connection, it has also found applications in the study of generic Newton points, cf. \cref{generic-ADLV-dim}. 

An interesting connection between the Demazure product on $\wtd{W}$ and the quantum Bruhat graph of $\Phi_W$ is established in \cite{HN21}. For any dominant $\underline\l \in X_*(T)_{\G_0}$, we define the depth
of $\underline\l$ to be $\text{depth}(\underline\l)=\min \{\<\a,\underline\l\>:\a\in \D\}$.

\begin{theorem}\cite[proposition 3.3]{HN21}\label{dem}
Let $x_1t^{\underline \mu_1}y_1, x_2t^{\underline \mu_2}y_2$ be two elements of $\wtd{W}$ such that the depth of $\underline \mu_1, \underline\mu_2$ is at least $2$. Then $\underline \mu_1+\underline \mu_2-\wt(y_1\i, x_2)$ is dominant, and
$$x_1t^{\underline \mu_1}y_1 \ast x_2t^{\underline \mu_2}y_2=x_1 t^{\underline \mu_1+\underline \mu_2-\wt(y_1\i, x_2)}y_2.$$
\end{theorem}

\subsection{The $\s$-conjugacy classes of $\breve G$} We say that two elements $b, b' \in \Breve{G}$ are $\s$-conjugate if there is some $g\in \Breve{G}$ such that $b'=g b \s(g) \i$. Let $B(\bG)$ be the set of $\s$-conjugacy classes on $\breve G$. By the work of Kottwitz in \cite{Ko85} and \cite{Ko97}, any $\s$-conjugacy class $[b]$ is determined by two invariants: 
\begin{itemize}
	\item The Kottwitz point $\k([b]) \in \pi_1(\bG)_{\G}$, the set of $\G$-coinvariants of the Borovoi fundamental group $\pi_1(\bG)$; 
	
	\item The Newton point $\nu([b]) \in ((X_*(T)_{\Gamma_0, \BQ})^+)^{\langle\sigma\rangle}$, the set of $\<\s\>$-invariants of the intersection of $X_*(T)_{\Gamma_0}\otimes \BQ=X_*(T)^{\Gamma_0}\otimes \BQ$ with the set $X_*(T)_\BQ^+$ of dominant elements in $X_*(T)_\BQ$.
\end{itemize}

We denote by $\le$ the dominance order on $X_*(T)_\BQ^+$, i.e., for $\nu, \nu' \in X_*(T)_\BQ^+$, we have $\nu \le \nu'$ if and only if $\nu'-\nu$ is a non-negative (rational) linear combination of positive coroots over $\breve F$. The dominance order on $X_*(T)_\BQ^+$ extends to a partial order on $B(\bG)$. Namely, for $[b], [b'] \in B(\bG)$, we say that $[b] \le [b']$ if and only if $\k([b])=\k([b'])$ and $\nu([b]) \le \nu([b'])$. 

Denote by ${\bf J}_b$ the  $\s$-centralizer group of $b$; this is an algebraic group over $F$ with $F$-rational points given by
\begin{equation*}\label{scentrali}
{\bf J}_b(F)=\{g\in \breve{G}\mid g^{-1}b\s(g)=b \}.
\end{equation*}
%Then ${\bf J}_b$ is an inner form of a $F$-Levi subgroup of $\bG$.
For any reductive group $\mathbf{H}$ over $F$, we denote by $\text{rk}_F^{\text{ss}} \mathbf{H}$ the semisimple $F$-rank of $\mathbf{H}$. The following result is implicit in \cite[\S 1.9]{Ko06}.

\begin{proposition}\label{rk}
Let $\bG$ be quasi-simple over $F$ and assume $\t \in \Omega$. Then $$\text{rk}_F^{\text{ss}} \mathbf{J}_{\dot \t}=\mid \text{Ad} \t \circ \s \text{ orbits on } \wtd{\BS} \mid -1.$$
\end{proposition}

\subsubsection{The straight $\s$-conjugacy classes}
Note that the action of $\s$ on $\breve{G}$ gives rise to an action on $\wtd{W}$, still denoted by $\s$. The set of $\s$-conjugacy classes in $\wtd{W}$ is denoted by $B(\wtd{W},\s)$.

Let $w\in \wtd{W}, n \in \mathbb{N}$. The \textit{$n$-th $\s$-twisted power of $w$} is defined by $w^{\s,n}=w\s(w)\s^2(w)\cdots \s^{n-1}(w)$. Then by definition, $w$ is called a \textit{$\s$-straight element} if $\ell(w^{\s,n})=n\ell(w)$, for all $n \in \mathbb{N}$. A $\s$-conjugacy class of $\wtd{W}$ is called straight if it contains a straight element, and we denote the collection of such straight $\s$-conjugacy classes by $B(\wtd{W},\s)_{\text{str}}$. We have a map $\Psi: B(\bG) \ra B(\wtd{W},\s)$, coming from the assignment $w \ra \dot{w}$. The importance of the straight $\s$-conjugacy class comes from the following result.
\begin{theorem}\cite[theorem 3.7]{He14}
The restriction of $\Psi$ induces a bijective map $\Psi: B(\wtd{W})_{\text{str}} \to B(\bG)$. Moreover, We have the following commutative diagram \[B(\wtd{W})_{\text{str}}\xymatrix{ \ar[rr]^{\Psi} \ar[dr]_{(\k,\nu)} & & B(\bG) \ar[ld]^-{(\k,\nu)} \\ & \pi_1(\bG)_{\s} \times ((X_*(T)_{\Gamma_0, \BQ})^+)^{\langle\sigma\rangle} &}\]
\end{theorem}

\subsubsection{The Newton poset} Given $w \in \wtd{W}$, define $$B(\bG)_w:=\{[b] \in B(\bG): [b] \cap \breve{I} w \breve{I} \neq \emptyset\}.$$ 
It is easy to see that $B(\bG)_w$ has an unique maximal element $b_w$, which coincides with the generic $\s$-conjugacy class in $\breve{I}w\breve{I}$. In the case of a split group, an explicit formula for the Newton point of this maximal element was first obtained in \cite{Mil21} under certain superregularity hypothesis on the dominant cocharacter associated with $w$. We need the following strengthened formula for the elements in shrunken chambers in the context of quasi-split groups, obtained in \cite{HN21}. We will write $\nu_w$ to denote $\nu(b_w)$.

\begin{theorem}\cite[proposition 3.1]{HN21}\label{newton}
Suppose that $\bG$ is quasi-split and adjoint over $F$. Let $x, y \in W$ and $\underline\mu \in X_*(T)_{\G_0}^+$ be such that $\text{depth}(\underline\mu)\geq 2$. Then $\nu_{x t^{\underline\mu} y}$ is the average of $\s$-orbit of $\underline\mu-\wt(y \i, \s(x))$. 
\end{theorem}
\subsubsection{Neutrally acceptable elements}\label{sec:bgmu}
Let $\{\mu\}$ be a conjugacy class of cocharacters over $\bar{F}$. Choose $\mu$ be a dominant representative of $\{\mu\}$ and denote by $\underline{\mu}$ its image in $X_*(T)_{\G_0}$. We define the \textit{$\mu$-admissible set}
\begin{equation*}
\Adm(\mu)=\{w \in \wtd{W};\ w \le t^{x(\underline \mu)} \text{ for some }x \in W\}.
\end{equation*}
We also have the set of \textit{neutrally acceptable elements}
\begin{equation*}
B(\bG, \mu)=\{ [b]\in B(\bG)\mid \kappa([b])=\mu^\natural, \nu([b])\leq \mu^\diamond \} .
 \end{equation*}
Here $\mu^\natural$ denotes the common image of $\mu\in\{\mu\}$ in $\pi_1(\bG)_{\s}$, and $\mu^\diamond$ denotes the average of the $\s_0$-orbit of $\underline \mu$. Here $\s_0$ denotes the $L$-action of $\s$, see \cite[definition 2.1]{GHN20} for details. The set $B(\bG, \mu)$ inherits a partial order from $B(\bG)$. Since the Kottwitz map $\k$ is constant on $B(\bG, \mu)$, we may view it as a subset of $X_*(T)_{\G_0,\BQ}$ via the
Newton map. The following description of $B(\bG, \mu)$ is obtained in \cite[theorem 1.1 \& lemma 2.5]{HN18}. To state the result, for each $\s_0$ orbit $\CO$ on $\BS$, we set $\varpi_{\CO}=\sum\limits_{i \in \CO} \varpi_i$.

\begin{theorem}\cite{HN18}\label{b-max}
(1) Let $v \in X_*(T)_{\G_0,\BQ}$. Then $v \in B(\bG, \mu)$ if and only if $\s_0(v)=v$ is dominant, and for any $\s_0$-orbit $\CO$ on $\BS$ with $\<v, \a_i\> \neq 0$ for each (or equivalently, some) $i \in \CO$, we have that $\<\mu+\s(0)-v, \varpi_{\CO}\> \in \BZ$ and $\<\mu-v, \varpi_{\CO}\> \geq 0$.

(2) The set $B(\bG, \mu)$ contains a unique maximal element.
\end{theorem}

\subsection{Affine Deligne-Lusztig varieties}
The \emph{affine Deligne-Lusztig variety} associated to  $w \in \tW$ and $b \in \breve G$ is a locally closed subscheme of the affine flag variety of $\bf G$, locally of finite type over $\overline{\mathbb F}_p$ and of finite dimension, with the set of geometric points given by
\begin{equation*}
X_w(b)(\overline{\mathbb F}_p)=\{g \breve I \in \breve G/\breve I\mid g \i b \s(g) \in \Breve{I} w \Breve{I}\}.
\end{equation*}
If $F$ is of equal characteristic, then by affine flag variety we mean the ``usual'' affine flag variety; in the case of mixed characteristic, this notion should be understood in the sense of perfect schemes, as developed by Zhu~\cite{Zhu17} and by Bhatt and Scholze~\cite{BS17}.

We now define
\begin{equation*}
X(\mu,b)=\bigcup\limits_{w \in \Adm(\mu)} X_w(b).
\end{equation*}
Settling the Kottwitz-Rapoport conjecture made in \cite{KR03} and \cite{Rap05} about the non-emptiness pattern for $X(\mu,b)$, He proves the following result in \cite{He16}.
\begin{theorem}\label{KR-conjecture}\cite[theorem A]{He16a}
$X(\mu,b) \neq \emptyset$ if and only if $[b] \in B(\bG,\mu)$.
\end{theorem}
Note that by \cref{KR-conjecture}, we have $$B(\bG,\mu) = \bigcup\limits_{w \in \Adm(\mu)} B(\bG)_w,$$ hence $b_{\max}=\max\limits_{w \in \Adm(\mu)}b_w$. Since $w \mapsto [b_w]$ is an increasing function $\wtd{W} \ra X_*(T)_{\G_0,\BQ}$, we conclude that $$b_{\max}=\max\limits_{x\in W}b_{t^{x(\mu)}}.$$

\section{Explicating $b_{\max}$ via generic Newton point}\label{case-free}
\subsection{Dimension of a generic ADLV}\label{sec:gendim}
Following \cite{He21}, we define the \textit{$n$-th $\s$-twisted Demazure power of $w$} by setting $$w^{\s,n}=w*\s(w)*\s^2(w)*\cdots *\s^{n-1}(w).$$
We need the following result, which on the one hand describes the dimension of a single affine Deligne-Lusztig variety associated with generic $\s$-conjugacy class, and on the other hand quantifies such generic $\s$-conjugacy class in terms of $\s$-twisted Demazure power. For the first equality below, see \cite[theorem 2.23]{He16} and \cite[lemma 3.2]{MV21}, whereas for the second equality, see \cite[theorem 0.1]{He21}.  
\begin{theorem}\label{generic-ADLV-dim}
Let $w \in \wtd{W}$. Then $\dim X_w(b_w)=\ell(w)-\<2\rho,\nu_w\>=\ell(w)-\lim\limits_{n \ra \infty} \frac{\ell(w^{*_\s,n})}{n}$.
\end{theorem}
Note that there can be elements $w\in \wtd{W}$ that are not pure translations but still satisfy $b_w=b_{\max}$; this can be seen e.g. using the Deligne-Lusztig reduction method, see \cite[proposition 4.2]{He14}. However, if $w$ contributes to top dimensional components - i.e., $\dim X(\mu,b_{\max})=\dim X_w(b_{\max})$ - then $w$ is necessarily a translation element of the form $t^{x\underline\mu}$ for some $x\in W$, by the first equality in \cref{generic-ADLV-dim}.

Now, again by an application of the first equality in \cref{generic-ADLV-dim} we see that $\dim X_{t^{x\underline\mu}}(b_{t^{x\underline\mu}})$ is a decreasing function of $b_{t^{x\underline\mu}}$; therefore, finding top dimensional $X_w(b_{\max})$ inside $X(\mu,b_{\max})$ boils down to understanding elements $x\in W$ such that $\dim X_{t^{x\underline\mu}}(b_{t^{x\underline\mu}})$ is minimized, and this minimum value of the dimension is indeed the dimension of $X(\mu,b_{\max})$. We approach the problem of computing $\dim X_{t^{x\underline\mu}}(b_{t^{x\underline\mu}})$ below in two different ways. In the rest of the paper, we omit the underline for simplicity and write $\mu$ instead of $\underline\mu$ throughout.

\subsection{A standard reduction} Let $\bG$ be a connected reductive group over $F$, and let $\bG_{\ad}$
be its adjoint group. Let $T_\ad$ be the image of $T$ in $\bG_\ad$, and denote by $\mu_\ad$ the image of $\mu$ in $X_*(T_\ad)_{\G_0}$. Similarly
for any $b \in \Breve{G}$, we denote by $b_\ad$ its image in $\Breve{G}_\ad$. By \cite[proposition 4.10]{Ko97}, we then have an isomorphism of posets $$B(\bG,\mu) \ra B(\bG_\ad,\mu_\ad), \text{~via~} [b] \mapsto [b_\ad].$$ Next, we have a decomposition
$\bG_\ad \simeq \bG_1 \times \cdots \times \bG_r$,
where each $\bG_i$
is adjoint and simple over F.  Write $\mu = (\mu_1,\cdots,\mu_r)$, where $\mu_i$
is a dominant coweight of $\bG_i$.
This way we can identify $$B(\bG_\ad,\mu)=\prod B(\bG_i,\mu_i).$$ We also have $$\dim X^\bG(\mu,b)=\dim X^{\bG_\ad}(\mu_\ad,b_\ad)=\sum \dim X^{\bG_i}(\mu_{\ad,i},b_{\ad,i}).$$
Hence, it suffices to focus on adjoint $F$-simple groups. We work with the absolute
local Dynkin diagram (i.e., the affine Dynkin diagram attached to $\bG$ over $\Breve{F}$), together with the diagram automorphism induced by Frobenius. 

We can describe such groups in terms of tuples $(\wtd{W}, \s)$, where $\wtd{W}$ is the associated Iwahori-Weyl group with its finite root datum being irreducible of adjoint type, and $\s=\text{Ad} (\t  \s_0)$ is a length preserving automorphism of $\wtd{W}$ with $\t \in \Omega, \s_0(\BS)=\BS$. Thus $\text{Ad} \t$ (resp. $\text{Ad}\s_0$) corresponds to a symmetry of the affine (resp. finite) Dynkin diagram. More concretely, $\t$ can be taken to be $\tau_i=t^{\varpi_i^\vee}w_{i,0}w_0$ whenever $\varpi_i^\vee$ is a minuscule fundamental coweight, and $\s_0$ can be nontrivial only when $W$ is of type $A_n, D_n (n \geq 5)$ and $E_6$. Throughout \cref{case-free} and \cref{b-max-compute}, we denote by $\z$ the finite Weyl group part of $\t$, i.e. $\z=\z_i:=w_{i,0}w_0$ for some minuscule fundamental coweight $\varpi_i^\vee$ - often suppressing the label $i$ as well. We henceforth assume that $\text{depth}(\mu) \geq 2$.
\subsection{via Demazure power}\label{sec:dem-power} Here we express $\<2\rho, \nu_{t^{x(\mu)}}\>$ explicitly using Demazure power of $t^{x(\mu)}$. 
\subsubsection{}
To ease notational burden, we separately discuss the case with trivial $\s_0$ first. Suppose that $o(\z)=m$. Then we have
\begin{equation*}
\begin{aligned}
       (t^{x\mu})^{*_\sigma,m+1}
       & =(x t^\mu x^{-1})*(\z x t^\mu x^{-1}\z^{-1})*\cdots *(\z^{m-1}xt^\mu x^{-1}\z^{1-m})*(xt^\mu x\i)\\
       & =xt^{\mu_1}x\i, \text{~with~} \mu_1:=(m+1)\mu-\sum\limits_{i=0}^{m-1} \wt(\z^ix,\z^{i+1}x).
\end{aligned}
\end{equation*}
Here the second equality follows from a repeated application of \cref{dem}. It is now easy to see that for any positive integer $k$, we have $$(t^{x\mu})^{*_\sigma,km+1}=xt^{\mu_k}x\i, \text{~with~} \mu_k:=(km+1)\mu-k\sum\limits_{j=0}^{m-1} \wt(\z^jx,\z^{j+1}x).$$ Note that $\mu_k$ is dominant under the depth hypothesis on $\mu$ and we have $\frac{1}{km+1}\ell((t^{x\mu})^{*_\sigma,km+1})=\<2\rho, \mu\>-\frac{k}{km+1}\<2\rho, \sum\limits_{j=0}^{m-1} \wt(\z^jx,\z^{j+1}x)\>$. Since the limit in \cref{generic-ADLV-dim} exists, we deduce 
\begin{equation*}\label{lim-cal1}
    \lim\limits_{n \ra \infty} \frac{1}{n}\ell((t^{x\mu})^{*_\s,n})=\lim\limits_{k \ra \infty} \frac{1}{km+1}\ell((t^{x\mu})^{*_\sigma,km+1})=\<2\rho, \mu\>-\frac{1}{m}\<2\rho, \sum\limits_{j=0}^{m-1} \wt(\z^jx,\z^{j+1}x)\>.
\end{equation*}
Thus by \cref{generic-ADLV-dim}, $\dim X_{t^{x\mu}}(b_{t^{x\mu}})=\frac{1}{m}\<2\rho, \sum\limits_{j=0}^{m-1} \wt(\z^jx,\z^{j+1}x)\>$.
\subsubsection{}
Now we consider the general case, where $\s_0$ may be nontrivial. Then $$\s^j(t^{x\mu})=\text{Ad}((\z \s_0)^j)(t^{x\mu}))=\text{Ad}(\z^{\s_0,j})(t^{\s_0^j(x\mu)}),$$ where we recall $\z^{\s_0,j}=\z\s_0(\z)\s_0^2(\z)\cdots \s_0^{j-1}(\z)$. Let $m'$ denote the order of $\z \s_0$ inside $W \rtimes \<\s_0\> \subset \Aut(W)$; since the group is adjoint, this is also the order of $\s$ as an automorphism of the subset of translation elements in $\wtd{W}$.

Then a similar calculation as above shows for any positive integer $k$, $$(t^{x\mu})^{*_\sigma,km'+1}=xt^{\mu_k'}x\i, \text{~with~} \mu_k'=\mu+k\sum\limits_{j=0}^{m'-1} \s_0^j(\mu)-k\sum\limits_{j=0}^{m'-1}\wt(\z^{\s_0,j}\s_0^j(x),\z^{\s_0,j+1}\s_0^{j+1}(x)).$$ In other words, the dominant translation part of $(t^{x\mu})^{*_\sigma,km'+1}$ is
\begin{equation}\label{mu-k}
    \mu_k':=\mu+km'\mu^\diamond-k\sum\limits_{j=0}^{m'-1}\wt(\z\s_0)^j(x),(\z\s_0)^{j+1}(x)).
\end{equation}

Noting that $\<2\rho,\mu\>=\<2\rho,\s_0^j(\mu)\>$ as $2\rho$ is $\<\s_0\>$-invariant, we then proceed in the same way and conclude $\dim X_{t^{x\mu}}(b_{t^{x\mu}})=\frac{1}{m'}\<2\rho, \sum\limits_{j=0}^{m'-1} \wt(\z\s_0)^j(x),(\z\s_0)^{j+1}(x))\>$. 

Rewriting the dimension formula obtained above using \cref{wt-d}, we deduce the following.
\begin{proposition}\label{dim-d-av}
Let $\bG$ be an adjoint $F$-simple group and $x\in W$ be an element in its Weyl group. Suppose that $\s=\text{Ad}(\t \s_0)$ is the Frobenius, with $\t \in \Omega$ and $\s_0(\BS)=\BS$. Let $\z$ be the finite Weyl group component of $\t$, and assume that $\text{depth}(\mu) \geq 2$. Finally, let $o_{\text{tr}}(\s)$ denote the order of $\s$ as an element of $\Aut(X_*(T)_{\G_0})$. Then $$\dim X_{t^{x\mu}}(b_{t^{x\mu}})=\frac{1}{o_{\text{tr}}(\s)}\sum\limits_{j=0}^{o_{\text{tr}}(\s)-1} d(\z\s_0)^j(x),(\z\s_0)^{j+1}(x)).$$
\end{proposition}

\subsection{via reduction to quasi-split case} 
Recall that $b_{\max} = \max\limits_{x \in W}  b_{t^{x(\mu)}}$. For computation of the generic Newton point, we can further pass on to a quasi-split group in the following way. For an adjoint $F$-simple group $\bG$ as in the previous subsection, let $\bG^*$ be its quasi-split inner form; then the Frobenius action for $\bG^*$ is given by $\s_{\bG^*}=\s_0$. By \cite[\S 2.4]{GHN16}, we can identify $B(\bG) = B(\bG^*)$ via $[b] \ra [b\dot{\t}]$, and this
bijection preserves the partial order defined via the closure relations in $\Breve{G}=\Breve{G}^*$. Hence, restricting this map to a natural bijection $B(\bG)_w = B(\bG^*)_{w\tau}$
for any $w \in \wtd{W}(\bG) = \wtd{W} (\bG^*)$, we then get $\nu^{\bG}([b_w])=\nu^{\bG^*}([b_w\dot{\t}])=\nu^{\bG^*}([b_{w\t}])$.

We now give alternative (in fact, simpler) formula for the dimension by relating it to the quasi-split case. 
\begin{proposition}\label{dim-qs}
In the setup of \cref{dim-d-av}, we have $\dim X_{t^{x\mu}}(b_{t^{x\mu}})=d(\z\i x, \s_0(x))$.
\end{proposition}
\begin{proof}
Assume first that $\text{depth}(\mu)\geq 3$. Note that $t^{x\mu}\t=t^{x\mu}t^{\varpi_i^\vee}w_{i,0}w_0=xt^{\mu+x\i \varpi_i^{\vee}}x\i w_{i,0}w_0$. Here $\mu+x\i\varpi_i^\vee$ is dominant with depth at least $2$: for any simple root $\a$, we have $$\<\mu+x\i\varpi_i^\vee, \a\>=\<\mu,\a\>+\<\varpi_i^\vee, x\a\>\geq \<\mu,\a\>+\<\varpi_i^\vee, -\theta\>=\<\mu,\a\>-1 \geq 2.$$ Hence, by \cref{newton} we have $\nu^{\bG}(b_{t^{x\mu}})=\nu^{\bG^*}([xt^{\mu+x\i \varpi_i^{\vee}}x\i \z])=(\mu+x\i \varpi_i^\vee-\wt(\z\i x, \s_0(x)))^\diamond$, giving us
\begin{equation}\label{interm1}
    \dim X_{t^{x\mu}}(b_{t^{x\mu}})=\<2\rho, \mu-\nu(b_{t^{x(\mu)}})\> = \<2\rho, \text{wt}(\z\i x, \s_0(x))-x\i \varpi_i\>.
\end{equation}
Note that we used $\<\s_0\>$-invariance of $2\rho$ above. Now using the length formula, we get $$\<2\rho,\mu\>=\ell(t^{x(\mu)})=\ell(t^{x(\mu)}\tau)=\ell(xt^{\mu+x\i \varpi_i^{\vee}}x\i \z)=\ell(x)+\<2\rho, \mu+x\i \varpi_i^\vee \> - \ell(x\i \z),$$ giving us $\<2\rho, x\i \varpi_i^\vee\>=\ell(x\i \z)-\ell(x)$. Plugging this in \cref{interm1} and applying \cref{wt-d}, we then get $$\dim X_{t^{x\mu}}(b_{t^{x\mu}})=\<2\rho, \text{wt}(\z\i x, \s_0(x))\>-\ell(x\i \z)+\ell(x)=d(\z\i x, \s_0(x)).$$ 
Combining \cref{two-expressions} and \cref{dim-d-av}, we see that this last formula is valid even under the weaker hypothesis that $\text{depth}(\mu) \geq 2$. Hence, we are done.
\end{proof}
Let us record an immediate corollary of \cref{dim-qs} and \cref{dim-d-av}.
\begin{corollary}\label{two-expressions}
Let $x \in W$. Suppose that $\s_0$ is an automorphism of $W$ with $\s_0(\BS)=\BS$, and $\z=\z_i=w_{i,0}w_0$ for some minuscule fundamental coweight $\varpi^\vee_i$. Finally, let $o_{\text{tr}}(\s)=\min\{j: \z^{\s_0,j}=1\}$. Then $$\frac{1}{o_{\text{tr}}(\s)}\sum\limits_{j=0}^{o_{\text{tr}}(\s)-1} d(\z\s_0)^j(x),(\z\s_0)^{j+1}(x))=d(\z\i x, \s_0(x)).$$
\end{corollary}
\subsection{Finishing the proof}\label{sec:finish}
As an consequence of the dimension formulas in the preceeding sections, we obtain the following estimate. 
\begin{proposition}\label{bounds}
In the setup of \cref{dim-d-av}, we have $$\ell(x\i \z \s_0(x))\geq \dim X_{t^{x\mu}}(b_{t^{x\mu}}) \geq \ell_R(x\i \z \s_0(x)).$$
\end{proposition}
\begin{proof}
We can prove it using the dimension formula established in either \cref{dim-d-av} or \cref{dim-qs}, but we choose to resort to the latter for notational ease. Then the assertion follows from the easy observation:

(a) for any two elements $a,b \in W$, we have $\ell(a\i b) \geq d(a,b) \geq \ell_R(a\i b)$.

Taking any reduced expression for $a\i b=s_{i_1}\cdots s_{i_k}$ and following the edges determined by the simple roots $\a_{i_k},\cdots, \a_{i_1}$ in order, we obtain a path from $a$ to $b$ which
has exactly $\ell(a\i b)$ edges. This proves the first inequality. For the other one, let $d=d(a,b)$ and choose a shortest path $a \ra ar_1 \ra ar_1 r_2 \cdots \ra ar_1\cdots r_d=b$, where $r_i$'s are not necessarily simple roots. Then $a\i b = r_1\cdots r_d$ is a decomposition of $a\i b$ into reflections, proving the other inequality.
\end{proof}
The reflection length of $w \in W$ is the smallest non-negative integer $l$ such that $w$ can
be written as a product of $l$ reflections in $W$. We denote by $\ell_R(w)$ the reflection
length of $w$. Let $\xi \in W$. Then we denote the $\s_0$ conjugacy class of $\xi$ in $W$ by $\CO_{\s_0}(\xi)$. We define $\ell_R(\CO_{\s_0}(\xi))=\min\{\ell_R(w): w\in \CO_{\s_0}(\xi)\}$. We can define $\ell(\CO_{\s_0}(\xi))$ similarly.

\begin{proof}[Proof of \cref{thm}]
Since $\dim X(\mu,b_{\max}) = \min\limits_{x \in W} \dim X_{t^{x\mu}}(b_{t^{x\mu}})$, we see from \cref{bounds} $$\ell(\CO_{\s_0}(\z)) \geq \dim X(\mu,b_{\max}) \geq \ell_R(\CO_{\s_0}(\z)).$$
By \cref{key-lemma} proved in \cref{sec:minlength}, we have $\ell(\CO_{\s_0}(\z))=\ell_R(\CO_{\s_0}(\z))=|\s_0\backslash \wtd{\BS}|-|\s \backslash \wtd{\BS}|$, hence it equals the dimension. By \cref{rk}, $|\s_0\backslash \wtd{\BS}|-|\s \backslash \wtd{\BS}|=\text{rk}_F^{\text{ss}} \bG^* - \text{rk}_F^{\text{ss}} \bG$, and we are done.
\end{proof}
\begin{remark}
As a byproduct of our discussion, we obtain that $$\ell_R(\mathcal(O_{\s_0}(\z))=\min\limits_{x\in W} d(\z\i x,\s_0(x))$$ for the element $\z=w_{i,0}w_0$, whenever $\varpi_i^\vee$ is minuscule. Contrast this with theorem 5.1 in \cite{HY21} that says $$\ell_R(\mathcal{O}_{\s_0}(w))=\min\limits_{x \in W} d(x,\s_0(x)w)$$ for the element $w=w_0$. It would be an interesting problem to explore connection between reflection length and distance in the quantum Bruhat graph for arbitrary elements of $W$.
\end{remark}

\section{Minimal length elements in certain $\s_0$-conjugacy class}\label{sec:minlength}
For simplicity, in the rest of the paper,
we will say $\s_0$-orbits in $\wtd{\BS}$ instead of $\text{Ad}(\s_0)$-orbits, likewise for $\text{Ad}(\s)$ etc. The goal of this section is to prove the following result via a case-by-case analysis. 
\begin{lemma}\label{key-lemma}
Suppose that $(W,\BS)$ is an irreducible finite Coxeter system associated to a Weyl group $W$ and let $\wtd{\BS}$ be the associated set of affine simple roots. Let $\s_0$ be a diagram automorphism of $(W,\BS)$ and $\t=t^{\varpi_i^\vee}w_{i,0}w_0\in \wtd{W}$, where $\varpi_i^\vee$ is a minuscule fundamental coweight. Then there exists a subset $J \subset \BS$ of cardinality $|\s_0\backslash \wtd{\BS}|-|\t \s_0 \backslash \wtd{\BS}|$ such that the set $\CO_{\min,\s_0}(w_{i,0}w_0)$ of minimal length elements in $\CO_{\s_0}(w_{i,0}w_0)$ consists of Coxeter elements of $W_J$. This subset $J$ is unique, unless $(W,\varpi_i^\vee)$ is of the form $(^2 D_{2n},\varpi_{2n-1}^\vee)$ or $(^2 D_{2n},\varpi_{2n}^\vee)$.
\end{lemma}
\begin{proof}
As our calculation below shows, the set $J$ is either empty or $W_J$ has only one Coxeter element in the cases where $\s_0$ is nontrivial. As for the cases where $\s_0$ is trivial, we only have to show \text{some} Coxeter element from asserted $W_J$ belongs to $\CO_{\min,\s_0}(w_{i,0}w_0)$ - since all such Coxeter elements are conjugate to each other. 

We follow the labelling of \cite{Bou}. To simplify notation, in the exceptional types we may simply write $s_{ij\cdots}$ for $s_i s_j \cdots$.
For $1 \leq a \leq b\leq n$, we abbreviate $s_{[a,b]}=s_as_{a+1}\cdots s_b$. Below we group together the cases according to their absolute local Dynkin diagram.
\begin{enumerate}
    \item Type $A_n$: Here all the fundamental coweights are minuscule and $\t_i$ acts on $\wtd \BS$ by $i$-step rotations, thus the number of $\t_i$-orbits on $\wtd{S}$ are $\frac{|\wtd \BS|}{|o(\text{Ad~}\t_i)|}=\frac{n+1}{\gcd(n+1,i)}$. 
    
    Here $\z_1=s_{[1,n]}$ is a Coxeter element for $W$ and $\z_i=\z_1^i$ for $i \geq 2$. Set $\k_i=\frac{n+1}{\gcd(n+1,i)}$ and let $m$ be the largest integer smaller than $\frac{n}{\k_i}$, i.e. $m=\lfloor \frac{n}{\k_i} \rfloor$. Then we have $$J_i= \{1,\cdots,\k_i-1\} \cup \{\k_{i+1},\cdots, 2\k_i-1\}\cup \cdots \cup \{m\k_i+1,\cdots, n\}.$$
    
    \item Type $^2 A_n$: The nontrivial diagram automorphism $\s_0$ of $\BS$ is conjugation by $w_0$, hence minimal length elements in $\CO_{\s_0}(w_{i,0}w_0)$ are maximal length elements in $\CO_{\text{id}}(w_{i,0})$ multiplied with $w_0$ on the right. The orbits of $\s_0$ on $\wtd{\BS}$ are 
    \begin{equation*}
    \begin{cases}
     \{0\}, \{i,n+1-i\} \text{~for~} 1\leq i \leq \frac{n}{2}  & ,\text{if $n$ is even;}\\
     \{0\},\{\frac{n+1}{2}\}, \{i,n+1-i\} \text{~for~} 1\leq i \leq \frac{n-1}{2} & ,\text{if $n$ is odd.}
    \end{cases}   
    \end{equation*}
    In this case, $\t_i \s_0$-orbits on $\wtd{\BS}$ are $\{0, i\}, \{j, i - j\}$ for $0 < j < i$,
    and $\{i + j, n - j\}$ for $0 < j < n - i$. We have that $J=\emptyset$ if $n$ is even or $i$ is even, and $J=\{ \frac{n+1}{2}\}$ if both $n, i$ are odd.
    \item Type $B_n$: Here the only minuscule coweight is $\varpi_1^\vee$, so $\t=\t_1$; its orbits are $\{0,1\},\{i\}$ for $2\leq i \leq n$. We have $\z_1=s_{[1,n-1]}s_n s_{[1,n-1]}\i$, thus $J=\{n\}.$
    \item Type $C_n$: Here the only minuscule coweight is $\varpi_n^\vee$, so $\t=\t_n$; its orbits are
    \begin{equation*}
    \begin{cases}
     \{i,n+1-i\}, \text{~for~} 0 \leq i \leq \frac{n-1}{2} & ,\text{if $n$ is odd;}\\
     \{\frac{n}{2}\}, \{i,n+1-i\}, \text{~for~} 0 \leq i \leq \frac{n}{2}-1 & ,\text{if $n$ is even.}
    \end{cases}   
    \end{equation*}
    We have $\z_n=s_n s_{[n-1,n]}\cdots s_{[2,n]}s_{[1,n]}$ and $J=\{i: 1\leq i \leq n, i \text{~is odd}\}$.
    \item Type $D_n$: Here we have three minuscule coweights: $\varpi_1^\vee, \varpi_{n-1}^\vee, \varpi_n^\vee$. There is an outer diagram automorphism of $D_n$ permuting the last two coweights. Thus it suffices to consider the case where $\t=\t_1$ or $\t_n$.
    
    Case 1: $\t=\t_1$. The $\t$-orbits on $\wtd{\BS}$ are $\{0, 1\}, \{n - 1, n\}$ and $\{i\}$ for $2 \leq i \leq n - 2$. We have $\z=\z_1=s_{[1,n-2]} s_{[1,n]}\i$ and $J=\{n-1,n\}$.
    
    Case 2: $\t=\t_n$; then there are two sub-cases.
    \begin{itemize}
        \item $n$ is odd: The $\t$-orbits on $\wtd{\BS}$ are $\{0, n, 1, n-1\}$ and $\{i, n-i\}$ for $2\leq i \leq \frac{n-1}{2}$. Here $\z=\z_n=s_ns_{[n-2,n-1]}s_{[n-3,n-2]}s_n\cdots s_ns_{[3,n-1]}s_{[2,n-2]}s_ns_{[1,n-1]}$ and $J=\{i: 1\leq i\leq n-2, i\text{~is odd}\} \cup \{n-1,n\}$.
        \item $n$ is even: The $\t$-orbits on $\wtd{\BS}$ are $\{i,n-i\}$ for $0 \leq i \leq \frac{n}{2}$. Here $\z=\z_n= s_ns_{[n-2,n-1]}s_{[n-3,n-2]}s_n$ $\cdots s_ns_{[2,n-1]}s_{[1,n-2]}s_n$ and 
        \begin{equation*}
        J=
    \begin{cases}
     \{i: 1\leq i \leq n, i \text{~is odd}\}, & ,\text{if $n$ is congruent to $0$ mod $4$;}\\
     \{i: 1\leq i \leq n-2, i\text{~ is odd}\} \cup \{n\} & ,\text{if $n$ is congruent to $2$ mod $4$.}
    \end{cases}   
    \end{equation*}
    \end{itemize}
    \item Type $^2 D_n$: The nontrivial diagram automorphism $\s_0$ of $\BS$ here swaps the vertices labeled $n-1, n$; the orbits on $\wtd{\BS}$ are therefore $\{n-1,n\}, \{i\}$ for $0\leq i\leq n-2$. As explained before, we have the following three cases.
    
    Case 1: $\t=\t_1$. The $\t \s_0$-orbits on $\wtd{\BS}$ are $\{0, 1\}$ and $\{i\}$ for $2 \leq i \leq n$. One may directly check that $1 \in \CO_{\s_0}(\z_1)$ and hence $J=\emptyset$.
    
    Case 2: $\t=\t_n$; then there are two sub-cases.
    \begin{itemize}
        \item $n$ is odd: The $\t \s_0$-orbits on $\wtd{\BS}$ are $\{i, n-i\}$ for $0\leq i \leq \frac{n-1}{2}$, and $J=\{i: 1\leq i\leq n-2, i\text{~is odd}\}$.
        \item $n$ is even: The $\t\s_0$-orbits on $\wtd{\BS}$ are $\{0,1,n-1,n\}$ and $\{i,n-i\}$ for $2 \leq i \leq \frac{n}{2}$, and the choices for $J$ are $\{i: 1\leq i \leq n, i \text{~is odd}\}$ and $\{i: 1\leq i \leq n-2, i\text{~is odd}\} \cup \{n\}$.
    \end{itemize}
    
    \item Type $^3 D_4$: Without loss of generality, we may assume that $\s_0$ is the outer diagram automorphism on $D_4$ sending $1 \ra 3 \ra 4 \ra 1$. As $\langle \s_0\rangle$ acts transitively on $\{1, 3, 4\}$, it suffices to consider the case where $\t=\t_1$. In that case, the orbits of $\t \s_0$ on $\wtd{\BS}$ are $\{0,1,4\},\{2\},\{3\}$. Finally, one directly checks that $1 \in \CO_{\s_0}(\z_1)$, whence $J=\emptyset$.
    \item Type $E_6$: There are two minuscule coweights: $\varpi_1^\vee, \varpi_6^\vee$. The unique outer diagram automorphism of $E_6$ permutes these two coweights. Thus it suffices to consider the case where $\t=\t_1$; its orbits on $\wtd{\BS}$ are $\{0,1,6\},\{2,3,5\},\{4\}$. We have $\z=\z_1=s_{1345624534132456}$, and $J=\{1,3,5,6\}$.
    \item Type $^2 E_6$: As explained above, it suffices to consider the case where $\t=\t_1$. Here $\s_0$ acts on $\wtd{\BS}$ by swapping $1 \leftrightarrow 6, 3 \leftrightarrow 5$, so its orbit on $\wtd{\BS}$ are $\{0\},\{1,6\},\{2\},\{3,5\},\{4\}$. Also, the orbits of $\t \s_0$ on $\wtd{\BS}$ are $\{0,1\},\{2,3\},\{4\},\{5\},\{6\}$. Finally, one directly checks that $1 \in \CO_{\s_0}(\z_1)$, whence $J=\emptyset$.
    \item Type $E_7$: There is a unique minuscule coweight: $\varpi_7^\vee$, so $\t=\t_7$. In this case, $\t$-orbits on $\wtd{\BS}$ are $\{0, 7\}, \{1, 6\}, \{3, 5\}, \{2\}, \{4\}$. Here $\z=\z_7=\z=s_{765432456713456245341324567}$, and $J=\{2,5,7\}$.

\end{enumerate}
\end{proof}
\begin{remark}
In type $A_n$, the elements $\z_i$ are regular elements of the Weyl group. In that case, the asserted conjugacy relations follow more conceptually from \cite[theorem 4.2(4)]{Sp74}.
\end{remark}
\section{Explicit description of the Newton point of $b_{\max}$}\label{b-max-compute}
Note that the existence of $b_{\max}$ as in \cref{b-max} was asserted in \cite{HN18} via reduction to a similar statement on the associated Iwahori-Weyl group, and then they further reduce it to treat the superbasic case and conclude it from there. However, this is somewhat indirect and does not give an explicit formula for the maximal element. Based on our calculation, in this section we list out $\nu(b_{\max})$ for the cases treated in \cref{key-lemma}. 

\subsection{}
Recall that the calculation in \cref{case-free} combined with \cref{generic-ADLV-dim} gives us
\begin{equation}\label{2rho}
    \< 2\rho, \nu_{t^{x\mu}} \>=\<2\rho,\mu^\diamond - \text{av}_{\s}(x)\>,
\end{equation}
where we set $\text{av}_{\sigma}(x):=\frac{1}{o_{\text{tr}}(\s)}\sum\limits_{i=0}^{o_{\text{tr}}(\s)-1}\wt((\z \s_0)^i(x),(\z \s_0)^{i+1}(x))$
However, one can upgrade this.
\begin{proposition}\label{cochar=} Suppose that we are in the setup of \cref{dim-d-av}. Let us further assume that $\nu_{t^{x\mu}}$ is regular. Then we have $$\nu_{t^{x\mu}} =\mu^\diamond -\text{av}_{\s}(x).$$
\end{proposition}
\begin{lemma}\label{str}
Let $w \in \wtd{W}$. Then there is a straight element $ w_s \in \wtd{W}$ such that $w_s \leq w$ and $\nu(w_s)=\nu_w$. 
\end{lemma}
\begin{proof}
This essentially follows from the arguments in \cite[section 2.2]{He21}. 
\end{proof}
\begin{proof}[Proof of \cref{cochar=}]
Let $w'=t^{x(\mu)}$ and choose a straight element $w_s$ corresponding to $w:=w'\t$ as given in \cref{str}. By properties of Demazure product and by definition of straight elements, 
\begin{equation}\label{eqn-ineqn}
    w_s^{\s_0,n-1} \leq w_s^{\s_0,n}=w_s^{*_{\s_0},n} \leq w^{*_{\s_0},n} \text{ for all } n\in \mathbb{N}.
\end{equation}
Note that $w'^{*_\s,n}=w^{*_{\s_0},n}$ whenever $(\t \s_0)^n=1$, i.e. $o_{\text{tr}}(\s)\mid n$. Furthermore, for such $n$ we have $\s^n(t^{x\mu})=t^{x\mu}, \s_0^n=1$, and hence
\begin{equation}\label{power-eqn}
   w^{*_{\s_0},n+1}=(w'\t)^{*_{\s_0},n}*w'\t=(t^{x\mu})^{*_\s,n}*t^{x\mu}\t=(t^{x\mu})^{*_\s,n+1}\t. 
\end{equation}

Let us recall the following well-known fact.

(a) Let $\text{pr}: \wtd{W } \ra X_*(T)_{\G_0}$ be the projection map, which sends any element $w$ to the unique dominant
coweight $\l$ with $w \in W t^\l W$. Assume that $w_1, w_2 \in \wtd{W}$ such that $w_1 \leq w_2$ in Bruhat order. Then $(\text{pr}(w_1))^+ \leq (\text{pr}(w_2))^+$ in the dominance order.
%Since $\nu(w_s)$ is assumed to be regular, we have by \cite[Proposition 2.4]{He14} that $w_s=yt^\g \s(y)\i$ for some $y \in W$. 

Recall now from \cref{sec:dem-power} that $o_{\text{tr}}(\s) \mid o(\s)$; let $o(\s)=\Xi o_{\text{tr}}(\s)$. Then the calculation in \cref{sec:dem-power}, combined with \cref{power-eqn}, shows that for $n=ko(\s)+1$ (with $k \in \mathbb{Z}_{\geq 0}$), we have $$w^{*_{\s_0},n}=xt^{\mu_k''}x\i \t=x^{\mu_k''+x\i \varpi_i^\vee}x\i \z, \text{ with }\mu_k''=\mu+ko(\s) \mu^\diamond-k\Xi\sum\limits_{j=0}^{o_{\text{tr}}(\s)-1}\wt(\z\s_0)^i(x),(\z\s_0)^{i+1}(x)).$$  

Since $\nu(w_s)$ is assumed to be regular, we have by \cite[proposition 2.4]{He14} that $w_s=yt^\g \s_0(y)\i$ for some $y \in W$. Using this in \cref{eqn-ineqn} with $n=ko(\s)+1$ and then applying fact (a), we thus obtain $$k o(\s)\nu(w_s) \leq \mu+x\i \varpi^\vee+ko(\s) \mu^\diamond-k\Xi\sum\limits_{j=0}^{o_{\text{tr}}(\s)-1}\wt(\z\s_0)^i(x),(\z\s_0)^{i+1}(x)).$$
This gives $$\nu_w=\nu(w_s) \leq \frac{1}{ko(\s)}(\mu+x\i \varpi^\vee)+ \mu^\diamond-\frac{1}{o_{\text{tr}}(\s)}\sum\limits_{j=0}^{o_{\text{tr}}(\s)-1}\wt(\z\s_0)^i(x),(\z\s_0)^{i+1}(x)).$$ Letting $k$ to be sufficiently large and noting that Newton points have bounded denominator, we obtain $\nu_w\leq  \mu^\diamond-\frac{1}{o_{\text{tr}}(\s)}\sum\limits_{j=0}^{o_{\text{tr}}(\s)-1}\wt(\z\s_0)^i(x),(\z\s_0)^{i+1}(x)).$ Hence by \cref{2rho}, we are done.

\end{proof}
\subsection{} To describe $\nu(b_{\max})=\max\limits_{x\in W}\nu_{t^{x\mu}}$, it suffices by \cref{cochar=} to tabulate 
\begin{equation}\label{cochar}
    \mu^\diamond - \nu(b_{\max})=\min\limits_{x \in W} \text{av}_{\s}(x),
\end{equation}
for the cases discussed in \cref{key-lemma}. Note that $(\z\s_0)^j(x))\i(\z\s_0)^{j+1}(x)=\s_0^j(x\i\z\s_0(x))$. By a combination of \cite[corollary 3.3]{Sad21} and \cite[lemma 3.1]{Sad21}, we have $\wt(\z\s_0)^j(x),(\z\s_0)^{j+1}(x)) \leq \wt(\s_0^j(x\i\z\s_0(x)),1)$, whence $$\mu^\diamond - \nu(b_{\max}) \leq \min\limits_{x \in W}  \frac{1}{o_{\text{tr}}(\s)}\sum\limits_{j=0}^{o_{\text{tr}}(\s)-1} \wt(\s_0^j(x\i\z\s_0(x)),1).$$ Thus $\mu^\diamond = \nu(b_{\max})$ whenever $1 \in \CO_{\min,\s_0}(\z)$. One can read off such cases from the list given in the proof of \cref{key-lemma}; those cases can also be directly seen to be corresponding to quasi-split groups. For the remaining cases, the proof of \cref{thm} shows that the minimum in \cref{cochar} is realized at some/any element $x\in W$ such that $x\i \z \s_0(x)$ is a partial Coxeter element as described in \cref{key-lemma}. This way, one can evaluate $\text{av}_{\s}(x)$ at such an element $x$ to get the answer.

Alternatively, we can describe this minimum by applying the criterion in \cref{b-max}. Let us call $\Xi_\s=\min\limits_{x \in W} \text{av}_{\s}(x)=\mu^\diamond - \nu(b_{\max})$. Note that, when $\sigma=\text{Ad}(\tau_i)$, the condition in \cref{b-max} translates to $$\Xi_\s=\min\{\xi:
    \<\varpi_i^\vee+\xi, \varpi_j\rangle \in \mathbb{Z}, \langle \xi, \varpi_j \rangle \geq 0, \text{ for all }j\}.$$
Thus, we can simply use expressions of fundamental coweights in terms of simple coroots coming from the inverse of the Cartan matrix and find this quantity. For a real number $m$, we set $\{m\}=m-\lfloor m \rfloor$.
\begin{enumerate}
    \item Type $A_n$: For $1 \leq i \leq n$, we have $\varpi_i^\vee=\sum\limits_{j=1}^n a_{ij}\a_j^\vee$, where
    \begin{equation}\label{cowt}
        a_{ij}=\min(i,j)-\frac{ij}{n+1}=\min(i,j)-\lfloor \frac{ij}{n+1} \rfloor - \{\frac{ij}{n+1}\}.
    \end{equation}
    Hence, $\Xi_\s=\sum\limits_{j=1}^n\{\frac{ij}{n+1}\}\a_j^\vee$.
    \item Type $B_n$: Here $\varpi_1^\vee=\sum\limits_{j=1}^{n-1} \a_j^\vee+\frac{1}{2}\a_n^\vee$, whence $\Xi_\s=\frac{1}{2}\a_n^\vee$.
    \item Type $C_n$: Here $\varpi_n^\vee=\sum\limits_{j=1}^n \frac{j}{2}\a_j^\vee$, whence $\Xi_\s=\sum\limits_{j \text{ odd}, 1\leq j \leq n } \frac{1}{2}\a_j^\vee$.
    \item Type $D_n$: Here the relevant coweights are $\varpi_1^\vee$ and $\varpi_n^\vee$. 
    
    We have $\varpi_1^\vee=\sum\limits_{j=1}^{n-2} \a_j^\vee+\frac{1}{2}(\a_{n-1}^\vee+\a_n^\vee)$, thus $\Xi_\s=\frac{1}{2}(\a_{n-1}^\vee+\a_n^\vee)$. 
    
    On the other hand, $\varpi_n^\vee=\sum\limits_{j=1}^{n-2}\frac{j}{2}\a_j^\vee+\frac{n-2}{4}\a_{n-1}^\vee+\frac{n}{4}\a_n^\vee$, whence 
    \begin{equation*}
        \Xi_\s=
    \begin{cases}
       & \sum\limits_{j \text{ odd}, 1\leq j \leq n-2 }\frac{1}{2}\a_j^\vee+(1-\{\frac{n-2}{4}\})\a_{n-1}^\vee+(1-\{\frac{n}{4}\})\a_n^\vee,\text{ if $n$ is odd;}\\
       & \sum\limits_{j \text{ odd}, 1\leq j \leq n-2}\frac{1}{2}\a_j^\vee+(1-\{\frac{n-2}{4}\})\a_{n-1}^\vee ,\text{ if $n$ is congruent to $0$ mod $4$;}\\
       & \sum\limits_{j \text{ odd}, 1\leq j \leq n-2}\frac{1}{2}\a_j^\vee+(1-\{\frac{n}{4}\})\a_n^\vee,\text{ if $n$ is congruent to $2$ mod $4$.}
    \end{cases}   
    \end{equation*}
    \item Type $E_6$: Here $\varpi_1^\vee=\frac{1}{3}(4\a_1^\vee+3\a_2^\vee+5\a_3^\vee+6\a_4^\vee+4\a_5^\vee+2\a_6^\vee)$, thus $\Xi_\s=\frac{1}{3}(2\a_1^\vee+\a_3^\vee+2\a_5^\vee+4\a_6^\vee)$.
    \item Type $E_7$: Here $\varpi_7^\vee=\frac{1}{2}(2\a_1^\vee+3\a_2^\vee+4\a_3^\vee+6\a_4^\vee+5\a_5^\vee+4\a_6^\vee+3\a_7^\vee)$, thus $\Xi_\s=\frac{1}{2}(\a_2^\vee+\a_5^\vee+\a_7^\vee)$.
\end{enumerate}

Finally, we have the following remaining cases where $\s_0$ is non-trivial: 
\begin{enumerate}
    \item Type $^2 A_n$: As per our discussion above, we can assume that both $n,i$ are odd. It is easy to see using \cref{cowt} that $\varpi_j+\varpi_{n+1-j} \in \BZ_{\geq 0}\D$ for all $1\leq j < \frac{n+1}{2}$. Then by \cref{b-max}, $\Xi_\s=\min\{\xi:
    \<\varpi_i^\vee+\xi, \varpi_{\frac{
    n+1}{2}}\rangle \in \mathbb{Z}, \langle \xi, \varpi_{\frac{n+1}{2}} \rangle \geq 0, \s_0(\xi)=\xi\}$, whence again by \cref{cowt} applied to $j=\frac{n+1}{2}$ we get $\Xi_\s=\frac{1}{2}\a_{\frac{n+1}{2}}^\vee$.
    \item Type $^2 D_n$: As per our discussion above, we only need to deal with the case when $\t=\t_n$. %Note that $\varpi_{n-1}=\sum\limits_{j=1}^{n-2}\frac{j}{2}\a_j+\frac{n}{4}\a_{n-1}+\frac{n-2}{4}\a_n$, and thus for the $\s_0$-orbit $\CO=\{n-1,n\}$, we have $\varpi_\CO=\sum\limits_{j=1}^{n-2} j\a_j+\frac{2n-2}{4}(\a_{n-1}+\a_n)$.
    Considerations similar to the case of type $D_n$ shows that if $\Xi_\s=\sum\limits_{j=1}^n x_j \a_j^\vee$, then $x_j=\frac{1}{2}$ for $1\leq j \leq n-2$; to find the rest of the coefficients, we have to use the conditions: $$\s_0(\Xi_\s)=\Xi_\s , \<\Xi_\s, \varpi_{n-1}+\varpi_n\>\geq 0, \<\varpi_n^\vee+\Xi_\s, \varpi_{n-1}+\varpi_n\> \in \BZ.$$
    This in turn yields $x_{n-1}=x_n, x_{n-1}+x_n \geq 0 $ and $x_{n-1}+x_n+\frac{2n-2}{4} \in \BZ$. From this, we deduce
    \begin{equation*}
        \Xi_\s=
    \begin{cases}
       & \sum\limits_{j \text{ odd}, 1\leq j \leq n-2 }\frac{1}{2}\a_j^\vee,\text{ if $n$ is odd;}\\
       & \sum\limits_{j \text{ odd}, 1\leq j \leq n-2}\frac{1}{2}\a_j^\vee+\frac{1}{4}(\a_{n-1}^\vee+\a_n^\vee),\text{ if $n$ is even.}
    \end{cases}   
    \end{equation*}
    
\end{enumerate}
\printbibliography[heading=bibintoc]

\end{document}